\documentclass[12pt,reqno]{amsart}

\usepackage[LGR,T1]{fontenc}
\usepackage[utf8]{inputenc}

\title[]{Revisiting Fortet's proof of existence of a solution to the Schrödinger system}
\author[]{Christian Léonard}
\date{August 2018}

\usepackage{amssymb, amsmath, amsfonts, latexsym, enumerate}
\usepackage[usenames,dvipsnames]{pstricks}
\usepackage{epsfig}
\RequirePackage[colorlinks,linkcolor=blue,citecolor=blue,urlcolor=blue]{hyperref}
\usepackage[english]{babel}

\newtheorem{theorem}[equation]{Theorem}
\newtheorem{lemma}[equation]{Lemma}
\newtheorem{proposition}[equation]{Proposition}
\newtheorem{corollary}[equation]{Corollary}

\newtheorem{hypotheses}[equation]{Hypotheses}

\theoremstyle{remark}
\newtheorem{remark}[equation]{Remark}
\newtheorem{remarks}[equation]{Remarks}

\numberwithin{equation}{section}

\newcommand{\RR}{\mathbb{R}}
\newcommand{\Rn}{\mathbb{R}^n}

\newcommand{\1}{\mathbf{1}}

\renewcommand{\ae}{\textrm{-}\mathrm{a.e.}}

\newcommand{\scal}{\!\cdot\!}

\DeclareMathOperator{\supp}{supp}

\DeclareMathOperator{\sol}{sol}

\newcommand{\boulette}[1]{$\bullet$\ Proof of #1.}
\newcommand{\Boulette}[1]{\par\medskip\noindent $\bullet$\ Proof of #1.}

\newcommand\Lim[1]{\lim_{#1\rightarrow\infty}}

\newcommand{\XX}{ \mathcal{X}}
\newcommand{\YY}{ \mathcal{Y}}
\newcommand{\XY}{\XX\times\YY}

\newcommand{\IX}{\int _{ \XX}}
\newcommand{\IY}{\int _{ \YY}}

\newcommand{\uu}{ \mathsf{u}}

\newcommand{\nn}{ \mathsf{n}}
\renewcommand{\aa}{ \mathsf{a}}
\newcommand{\bb}{ \mathsf{b}}

 \address{Modal’X, UPL, Univ Paris Nanterre, F92000 Nanterre France}
 \email{christian.leonard@parisnanterre.fr}
\begin{document}

\maketitle 
\tableofcontents

\begin{abstract} 
Since a few years, the Schrödinger problem captures the attention of a growing community of mathematicians interested in optimal transport problems. The first result of existence of a solution to this problem dates back  to  1940, when Fortet published an article on the subject. In these notes, 
Fortet's original proof of existence and uniqueness of a solution to the Schrödinger system is revisited.
\end{abstract}

In 1931, Schrödinger addressed the following problem in the articles 
\cite{Sch31,Sch32}. Let  $(\XX, \mathsf{m})$ and $(\YY, \mathsf{n})$ be two measure spaces where $ \mathsf{m}$ and $ \mathsf{n}$ are nonnegative measures.  A reference measure  on the product space $\XY$:
\begin{align}\label{eq-19}
\mathsf{p}(dxdy)=p(x,y)\, \mathsf{m}(dx)\mathsf{n}(dy)
\end{align}
is given, with $p:\XY\to[0, \infty)$  a nonnegative function which is interpreted as a Markov transition density. 
\\
The relative entropy of the probability $\pi$ on $\XY$ with respect to $ \mathsf{p}$ is defined by: $$H(\pi| \mathsf{p}):=\int _{ \XY} \log(d\pi/d \mathsf{p})\,d\pi.$$
Motivated by a natural problem in statistical physics, Schrödinger proposes to minimize $H(\pi| \mathsf{p})$ among all probability measures $\pi$ on $\XY$  with prescribed marginal measures $\mu(dx)$ and $\nu(dy).$ Otherwise stated, the Schrödinger problem is 
\begin{align}\label{eq-16}
H(\pi| \mathsf{p})\to \textrm{min};\quad \pi: \pi_\XX=\mu,\ \pi_\YY=\nu,
\end{align}
with $\pi _{ \XX}(dx):= \pi(dx\times\YY)$ and $\pi _{ \YY}(dy):= \pi(\XX\times dy).$ It is highly reminiscent of the Monge-Kantorovich optimal transport problem. For more details about this connection, one can have a look at the survey paper \cite{Leo12e}.
\\
Schrödinger informally argues that the unique solution $\pi$ of \eqref{eq-16} (if it exists) has  the same  form as $ \mathsf{p}.$ And therefore, the problem to be solved is: Find two nonnegative measures $\aa$ on $\XX$ and $\bb$  on $\YY$ such that
\begin{align*}
 \left\{ \begin{array}{lcl}
 \pi(dxdy)&=& p(x,y)\, \aa(dx)\bb(dy),\\
\pi _{ \XX}(dx)&=& \mu(dx),\\
 \pi_\YY(dy)&=&\nu(dy).
 \end{array}
 \right.
\end{align*}
 Computing the marginal measures of $\pi$, we arrive at the equivalent formulation
\begin{align}\label{eq-01}
\left\{
\begin{array}{ll}
 \displaystyle{\aa(dx) \IY p(x,y)\, \bb(dy)}&=\ \mu (dx),\\ \\
\displaystyle{\bb(dy)\IX p(x,y)\, \aa(dx)} &=\ \nu (dy).
\end{array}
\right.
\end{align}
which is called the Schrödinger system.

Schrödinger's informal derivation of the connection between \eqref{eq-16} and \eqref{eq-01} was done with $\XX=\YY=\Rn$ and $p$ a Gaussian kernel.
The first proof of existence of a solution to \eqref{eq-01} was done, under some additional hypotheses (see Theorems \ref{res-05} and \ref{res-23} below), by Fortet in a paper published in  1940  \cite{Fort40}, with  $\XX=\YY=\RR$,  $p$  positive and continuous, and $\mu,\nu$  absolutely continuous with continuous densities.

\subsection*{Motivation}

The recent paper  \cite{EP18} by Essid and Pavon offers a rewriting in English of Fortet's French-written article \cite{Fort40}, with $\XX=\YY=\Rn$ instead of $\RR$, some additional technical details which were missing in Fortet's proof and a better organization of the argument.  It is deliberately very close to the original article, even respecting the original notation. Reading \cite{EP18} which is an ode to the  contribution of Fortet, and then going back to  \cite{Fort40}, made me want to write these notes, to get a better understanding of Fortet's  arguments. I hope that this will be of some use for  someone else than myself.

\section{Main results}

We obtained  several extensions of  Fortet's results. Nevertheless, their proofs are essentially the same as Fortet's ones. The improvements are mainly in terms of regularity: we remove topological restrictions when measurability is enough, and in terms of abstract generality: we consider general spaces instead of $\XX=\YY=\RR.$  We also put some effort in clarifying Fortet's original argument: (i) measurability and nonnegative integrability issues are faced once for all at the beginning of the proof, (ii) a reduction of the problem at Proposition \ref{res-06} makes life easier for the rest of the article, (iii) a synthetic criterion is presented at Lemma \ref{res-01} and (iv) a simple trick to pass from almost everywhere to everywhere statements is given at Lemma \ref{res-10}.
%To fix the  framework, $\XX$ and $\YY$ are Polish spaces equipped with their respective Borel $ \sigma$-fields,  $ \mu$ and $\nu$ are Borel probability measures, $p$ is  jointly measurable on $\XY,$ and the unknowns $a$ and $b$ must be found among all Borel $ \sigma$-finite nonnegative measures.

\subsection*{First theorem}

In contrast with Fortet, for the  first result we only assume that $\XX$ and $\YY$ are  measurable spaces, $p$ is measurable and $\mu,\nu$ are generic probability measures.

\begin{theorem}\label{res-05}
Suppose that the transition density $p$ is positive: 
$p(x,y)>0,\ \forall x\in\XX, y\in\YY,$
 bounded:
$\sup _{ x\in\XX, y\in\YY} p(x,y)< \infty,$ and that 
\begin{align}\label{eq-29}
\IY \Big[\IX p(x,y)\,\mu(dx)\Big]^{ -1}\,\nu(dy)< \infty
\quad \textrm{or}\quad
\IX\Big[\IY p(x,y)\, \nu(dy)\Big] ^{ -1}\,\mu(dx)< \infty.
\end{align}
Then, the Schrödinger system \eqref{eq-01}  admits a unique  solution.
\end{theorem}

\begin{proof}
Introducing $\uu:= d\mu /d\aa,$ we see that \eqref{eq-01} is equivalent to
\begin{align}\label{eq-10}
\aa= \uu ^{ -1} \mu ,\quad \bb= \left(\IX p(x;\cdot) \uu(x) ^{ -1} \mu (dx)\right) ^{ -1}\, \nu ,
\end{align}
where $\uu:\XX\to(0, \infty)$ solves the functional equation
\begin{align}\label{eq-02}
u= \Phi[u],
\end{align}
with the Fortet mapping $\Phi$ defined by
\begin{align}\label{eq-28}
\Phi[u](x):=
	\IY p(x,y)\Big[\IX p(x',y)\,u(x') ^{ -1} \,\mu (dx')\Big] ^{ -1}\ \nu (dy) ,\qquad x\in\XX,\ \mu \ae
\end{align}
The result is a direct consequence of the uniqueness result:  Proposition \ref{res-20}, and the existence result:  Theorem \ref{res-11}, for the fixed point equation \eqref{eq-02}.
\end{proof}

\begin{remarks}\label{rem-01}\begin{enumerate}[(a)]
\item
The positivity of $p$ appears as a strong  irreducibility hypothesis once one looks at the measure $ \mathsf{p}$, see \eqref{eq-19}, as the law of a Markov chain with two instants $t=0$ and $t=1.$

\item
The uniqueness of the solution must be understood up to the  transformations $u\mapsto \kappa u$ for any $\kappa>0$ if one looks at \eqref{eq-02},  or $(\aa,\bb)\mapsto (\kappa ^{ -1} \aa, \kappa \bb)$ if one looks at \eqref{eq-01}. In fact, \emph{the} solution of \eqref{eq-01} is the measure  $\pi(dxdy)=p(x,y)\aa(dx)\bb(dy).$\end{enumerate}\end{remarks}

\subsection*{Second theorem}
It is an interesting variant of Theorem \ref{res-05} where the integral estimates \eqref{eq-29}   are replaced by the following

\begin{hypotheses}\label{hyp-02}
The set $\XX$ is a topological space equipped with its Borel $ \sigma$-field and there exists a nonempty compact set $K\subset\XX$   such that the following statements are satisfied.
\begin{enumerate}[(a)]
\item
For all $y\in\YY,$ $p(\cdot,y):\XX\to (0, \infty)$ is continuous on $K$. 
\item
There exist finitely many $x_j\in\XX$ and $c_j>0$, $1\le j\le J,$ such that 
\begin{align*}
\sup _{ x\in K} p(x,y)\le \sum	_{ j\le J} c_j p(x_j,y),\quad \forall y\in\YY.
\end{align*}
\end{enumerate}
\end{hypotheses}

\begin{remarks}\ 
\begin{enumerate}[(i)]
\item
Hypothesis \ref{hyp-02}-(b) implies that
for any nonnegative function $v$ on $\YY,$  if \\ $\IY p(x,y) v(y)\, \nn(dy)< \infty$ for all $x\in \XX,$ then $ \left\{ p(x,\cdot) v(\cdot); x\in K\right\} $ is uniformly $\nn$-integrable. 
\item
Next theorem is still valid under this wider assumption. However, we state it under the Hypothesis \ref{hyp-02}-(b) which is more tractable.
\item
For comparison, Fortet (essentially) requires that  $p$ is jointly continuous and satisfies the uniform integrability property of item (i)  \emph{for any} compact subset $K\subset \XX.$ 
\end{enumerate}
\end{remarks}

\begin{theorem}\label{res-23}
Suppose as in Theorem \ref{res-05} that $p$ is positive and bounded. Under  the additional Hypotheses \ref{hyp-02}, the Schrödinger system \eqref{eq-01}  admits a unique  solution.
\end{theorem}

\begin{proof}
Similar to Theorem \ref{res-05}'s proof, where Theorem \ref{res-11} is replaced by Proposition \ref{res-17}.
\end{proof}

\begin{corollary}\label{res-21}
The state spaces are $\XX=\YY=\Rn$. Suppose that the  transition density writes as
\begin{align*}
p(x,y)= \theta(|y-x|),\qquad x,y\in\Rn,
\end{align*}
where $|\cdot|$ is the Euclidean norm and the function $ \theta:[0, \infty)\to (0, \infty)$ is positive, bounded, continuous and there exists  some $L\ge 0$ such that $ \theta$ is  non-increasing in restriction to $[L, \infty).$ 
\\
Then, the Schrödinger system \eqref{eq-01} admits a unique solution.
\end{corollary}

\begin{proof}
Postponed at page \pageref{pc21}.
\end{proof}

\subsection*{A  variant of Theorem \ref{res-05}}

 Again $p$ is positive,  but no topology is required and the integral estimates \eqref{eq-29} of Theorem \ref{res-05} are replaced by the
\begin{hypotheses}\label{hyp-03}\ 
\begin{enumerate}[(a)]
\item
$p$ is measurable and positive everywhere.
\end{enumerate}
There exists a measurable positive function $U:\XX\to (0, \infty)$ such that
\begin{enumerate}[(b)]
\item
 $\Psi[U](y):=\IX p(x',y)U(x')^{ -1}\,\mu(dx')< \infty$ for all $y\in\YY$,  
 
 \item
 
 $\Phi[U](x):=\IY p(x,y) \Psi[U](y)^{ -1}\,\nu(dy)< \infty,$ for all $x\in\XX,$ (recall the Fortet mapping at \eqref{eq-28})
\item[(d)]
 and for any $x_o,$ there exist  two real numbers $r> 1,\ c>0,$ such that
\begin{align*}
\IY \left( \frac{p(x,y)}{p(x_o,y)}\right) ^{ r} p(x_o,y) \Psi[U](y) ^{ -1}\, \nu(dy)
\le c\,U(x)^r,
\quad \forall x\in\XX.
\end{align*}
\end{enumerate}
\end{hypotheses}

This criterion does not appear in Fortet's work, which is done by setting $U=\1.$

\begin{proposition}\label{res-22}
Under the Hypotheses \ref{hyp-03},  the Schrödinger system \eqref{eq-01} admits a unique solution.
\end{proposition}

\begin{proof}
Similar to Theorem \ref{res-05}'s proof, where Theorem \ref{res-11} is replaced by Proposition \ref{res-18}.
\end{proof}

\subsection*{Outline of the rest of the article}

The remainder of this article is dedicated to the proofs of the existence and uniqueness of a positive solution to Fortet's equation \eqref{eq-02}. Uniqueness is stated at Proposition \ref{res-20};   existence is stated at Theorem \ref{res-11} and  its variants: Propositions \ref{res-17} and  \ref{res-18}. Next section proposes a short illustration, an informal presentation of Fortet's proof is given at Section \ref{sec-informal}, we gather preliminary material at Section \ref{sec-prelim}, the main results are proved at Section \ref{sec-main} and their variants are investigated in the final sections.

\section{A Gaussian example and some comments}

\subsection*{A Gaussian example}

Right after  Schrödinger addressed  Problem \eqref{eq-16} and proposed the system of equations \eqref{eq-01}, Bernstein studied it in \cite{Bern32} and claimed (without proof) that he solved \eqref{eq-01} with: $\XX=\YY=\Rn,$ $p$ a Gaussian kernel and $\mu,$ $\nu$ Gaussian measures. 
\\
Let 
\begin{align*}
n_\kappa(z):= \sqrt{\det\kappa/ (2\pi) ^n}  \exp \left( - {z\cdot\kappa z}/{2}\right),\qquad z\in\Rn,
\end{align*}
be the density of the Gaussian probability measure with variance $\kappa^{ -1},$ where $\kappa$ is a positive definite symmetric matrix ($\kappa>0$).  We assume that 
\begin{align*}
p(x,y)=n_c(y-x),\quad \mu(dx)=n_a(x)\, dx\quad \textrm{and}\quad \nu(dy)=n_b(y)\, dy,
\end{align*}
where $a,b,c>0$  are  positive definite symmetric matrices, not to be confused with the unknown  measures $\aa$ and $\bb$  of the Schrödinger system.

\begin{corollary}\label{res-25}
In this Gaussian setting, the Schrödinger system admits a unique solution for any positive $a,b$ and $c$.
\end{corollary}

\begin{proof}
Immediate consequence of Corollary \ref{res-21}.
\end{proof}

\subsection*{Testing Theorem \ref{res-05} and Proposition \ref{res-22}}

Now, let us try to prove Corollary \ref{res-25} by means of Theorem \ref{res-05} and Proposition \ref{res-22}. \emph{We shall observe that the integral criteria of these results do not permit us to recover Corollary \ref{res-25} in full generality.}

\begin{enumerate}[(i)]
\item
(Theorem \ref{res-05}). We have to evaluate the quantities in \eqref{eq-29}.  
It is an easy matter to compute the Laplace transform
\begin{align*}
\int _{ \Rn} e ^{ \lambda\cdot z}\,n_\kappa(z)\, dz= \exp \left( \lambda\cdot \kappa ^{ -1} \lambda/2\right)\propto n _{ - \kappa ^{ -1}}( \lambda) ,\qquad \lambda\in\Rn,
\end{align*}
which is the heart of Gaussian calculus and implies, regardless of multiplicative constants, that
\begin{align*}
\int n_c(y-x) n _{ \alpha}(x)\, dx \propto n _{  \alpha c( \alpha+c) ^{ -1}}(y).
\end{align*}
In particular, 
$\int p(x,y)\, \mu(dx)\propto   n _{  \beta}(y) 
$
with $ \beta=ac(a+c) ^{ -1},$
and\\
$\IY [\IX p(x,y)\, \mu(dx)] ^{ -1}\, \nu(dy)\propto \IY n _{ b- \beta}(y)\, dy$ which is finite if and only if $b- \beta>0$ (in the sense of quadratic form), that is: $ab+bc-ac>0.$  Inverting the role of $\mu$ and $\nu,$ we see that  \eqref{eq-29} is equivalent to 
\begin{align*}
ab+cb-ac>0\quad \textrm{or}\quad ab+ac-cb>0.
\end{align*}
It is enough for instance that $a\ge b$ or $b\ge a$ for this to be true.
In particular,    when $a$ and $b$ are multiples of the identity (for instance in dimension $n=1$), this is always satisfied.
{However, in the general case, it may happen that none of  these inequalities are verified}.

\item (Proposition \ref{res-22}). 
Choosing $U=\1$ and any $r>1$ drives us back more or less to (i), which is unsatisfactory.   \\
Let us take something  different, but still allowing for Gaussian calculus: $$U(x)= \exp (x\cdot dx/2)\propto n _{ -d}(x),$$ with $d$ some unsigned symmetric matrix. \emph{Again we show, by going as far as possible,  that this does not permit to prove that the Schrödinger problem admits a solution in all the  cases.}
\\
%The problem can be reduced by taking $c=1(=\Id)$ since the linear changes of variable $x'= c ^{ 1/2}x$ and $y'=c ^{ 1/2}y$ transform the transition density $n_c$ into $n _{ 1}$ and the constant Jacobian determinant of this linear transformation  does not break the product form $p(x,y) a(dx)\bb(dy)$.  Let us take $c=1$ from now on.\\
For any $y,$ 
\begin{align*}
\Psi[U](y)\propto\int n_c(y-x) n_d(x)n_a(x)\,dx
	\propto n _{ c(a+d)(a+c+d) ^{ -1}}(y)
\end{align*}
is finite when $c(a+d)(a+c+d) ^{ -1}>0,$ that is
\begin{align}\label{eq-32}
(i)\ d>-a,\qquad \textrm{or}\qquad (ii)\ d<-a-c.
\end{align} 
%It is also finite when $d$ is negative enough, but we won't consider this case
For any $x_o,x,$
\begin{align*}
&\IY \left( \frac{p(x,y)}{p(x_o,y)}\right) ^{ r} p(x_o,y) \Psi[U](y) ^{ -1}\, \nu(dy)\\
	\propto\ \ &n _{ rc}(x))\int \exp\Big((rc(x-x_o)+cx_o)\cdot y\Big)
		n _{ b+c-c(a+d)(a+c+d) ^{ -1}}(y)\,dy\\
	\propto\ \ & n _{ rc-r^2c^2[b+c-c(a+d)(a+c+d) ^{ -1}] ^{ -1}}(x).
\end{align*}
The quantity $\Phi[U](x)$ corresponds to $r=1.$ It is finite when
\begin{align}\label{eq-30}
1-c[b+c-c(a+d)(a+c+d) ^{ -1}] ^{ -1}>0,
\end{align}
and Hypothesis \ref{hyp-03}-(c) holds when 
$	%\begin{align*}
rc-r^2c^2[b+c-c(a+d)(a+c+d) ^{ -1}] ^{ -1}>-rd,
$	%\end{align*}
or equivalently
\begin{align}\label{eq-31}
\begin{split}
P_r(d):=d^2&+[a+2c-(r-1)b ^{ -1}c^2]d\\&+c[a+c-rab ^{ -1}c-(r-1)b ^{ -1}c^2]>0.
\end{split}
\end{align}
Taking \eqref{eq-32}-(i): $d>-a,$ into account,  \eqref{eq-30} is equivalent to
\begin{align}\label{eq-33}
(b-c)d>ac-ab-bc.
\end{align}

\begin{itemize}
\item
If $b>c,$ \eqref{eq-33} is satisfied for  all large enough $d$. Clearly,  \eqref{eq-31} is also verified for any $r>1$  when $d$ is large enough.

\item
If $b=c,$ a direct inspection shows that \eqref{eq-30} is verified for all $d>-a,$ and \eqref{eq-31} is also verified for any $r>1$  when $d$ is large enough. 
 \item
 If $b<c,$ the combination of \eqref{eq-32} and \eqref{eq-33} is
 \begin{align*}%\label{eq-34}
 -a<d<-a+bc(c-b) ^{ -1}=:\bar d.
 \end{align*}
The  polynomial $P_r$ defined at \eqref{eq-31} writes as 
\begin{align*}
P _{ r}( d)=P _{ r=1}( d)-(r-1)b ^{ -1}c^2(a+c+ d),
\end{align*}
and, with $r=1$,  its value at   the right bound $\bar d$ of the above interval  is  
\begin{align*}
P _{ r=1}(\bar d)=b ^{ -1}(c-b) ^{ -2}c^3(ab-ac+bc).
\end{align*}

 \begin{itemize}%[(i)]
 \item
 If $ab-ac+bc>0$, then $\bar d=(c-b) ^{ -1}(ab-ac+bc)>0$ and   $P _{ r=1}(\bar d)>0$. By continuity, choosing $d$ a little below $\bar d$ and $r$ a little above 1 yields \eqref{eq-31}.
 
 \item
  If $ab-ac+bc< 0$, then $\bar d=(c-b) ^{ -1}(ab-ac+bc)< 0$ and unfortunately $P_r(d)<0$ for all $d\in(-a,\bar d)$ and all $r\ge 1.$

 \end{itemize}
% Let us take $r>1$  close to 1 and take $d$ close to the bound $-a$ or $-a+b(1-b) ^{ -1}.$ 
 %With $r=1$, the value of polynomial \eqref{eq-31} at $d=-a$ is: $-a+1-ab ^{ -1}.$ It is positive if $ab+a-b<0.$\\
\end{itemize}
It remains to look at this last situation: $b<c, ab-ac+bc<0,$ in the domain \eqref{eq-32}-(ii): $d<-a-c.$ In this case, \eqref{eq-30} implies that
$(a+c+d) ^{ -1}c^2>c-b,$ which is impossible because $a+c+d<0$ and $c-b>0.$ 
\\
Exactly as in (i), we cannot conclude when $ab-ac+bc<0.$  
\end{enumerate}

It appears that  Theorem \ref{res-05} and Proposition \ref{res-22} are less efficient than Theorem \ref{res-23} and its Corollary \ref{res-21}. The winning trick of the proof of Theorem \ref{res-23} is Fortet's idea of twisting the transition density $p$ as described at page \pageref{twist}. 

\subsection*{Personal comments}

In conclusion, my feeling is that, although its proof is beautiful, Theorem \ref{res-05} is less powerful than the results obtained by means of the entropy method which was put forward by Beurling in  \cite{Beu60} to solve Schrödinger system. However, note that Beurling's results and its subsequent improvements, see for instance the survey article \cite{Leo12e} and the references therein, also require   integrability controls (in terms of relative entropies) which might be very difficult to verify in concrete situations. 

This last remark strongly makes the case for Theorem \ref{res-23}, since it does not require such integrability estimates. Improvements of Fortet's method should focus in this direction.

\section{Informal presentation of Fortet's proof}
\label{sec-informal}

\subsection*{A recurrence scheme}

Fortet solves the fixed point equation \eqref{eq-02} by means of a convergent scheme $(u_n)$  such that $u_*:=\Lim n u_n$ satisfies
\begin{align}\label{eq-06}
u_*= \Phi[u_*].
\end{align}
Clearly  $\Phi$ is  positive and increasing: if $0\le u\le u',$ then $0\le  \Phi[u]\le \Phi[u'].$ This will be used in a crucial manner in all what follows. 
\\
In particular, if the initial data ${\tilde u}_1\ge 0$ of the scheme 
\begin{align}\label{eq-20}
{\tilde u}_{ n+1}:=\Phi[{\tilde u}_n],\ n\ge1,
\end{align} 
is such that $\tilde u_2:=\Phi({\tilde u}_1)\le {\tilde u}_1,$  the decreasing  nonnegative sequence $({\tilde u}_n)$ admits a  limit which seems to be a reasonable candidate for a solution to \eqref{eq-02}.  But in general we do not have $\Phi(u)\le u$ for a generic $u\ge 0.$
\\
On the other hand, since $\uu:=d\mu/d\aa $ is positive and finite (see \eqref{eq-07} at Proposition \ref{res-06} for a justification), we look for \emph{a solution $u_*$ which is  positive and finite}. But $\Phi$ being positive 1-homogeneous, it admits $0$ and $ \infty$ as trivial fixed points, which must be avoided.
\\
These considerations suggest that the natural  scheme \eqref{eq-20} should be modified as follows:
\begin{align}\label{eq-03}
u _{ n+1}:= U/(n+1)\vee \Phi[u_n]\wedge U,\qquad n\ge1,
\end{align}
starting at $n=1$ from a positive function $u_1=U:\XX\to(0 ,\infty)$   to be chosen adequately.

Together with the basic inequality $u_2:= U/2\vee \Phi[u_1]\wedge U\le U= u_1$, we obtain $u_3:= U/3\vee \Phi[u_2]\wedge U\le U/2\vee  \Phi[u_1]\wedge U=:u_2,$ and so on by recurrence:
\begin{align}\label{eq-04}
u_{ n+1}\le u_n,\quad \forall n\ge 1,
\end{align}
with
\begin{align*}
0<U/n\le u_n\le U< \infty,\quad \forall n\ge 1.
\end{align*}
As a decreasing nonnegative sequence, $(u_n)$ admits a pointwise limit $$u_*(x):=\Lim n u_n(x),\quad x\in\XX,$$ such that $0\le u_*\le U< \infty.$ 

We observe that the introduction of $\wedge\,U$ with $u_1=U< \infty$ forces both the desired ignition  inequality $u_2\le u_1$ and the finiteness of the scheme. On the other hand, the introduction of the operation $U/(n+1)\,\vee$ implies that $u_n$ is positive for all $n\ge 1,$ which is necessary (but not sufficient) to obtain the positivity of $u_*.$ 

Fortet proved that, under some hypotheses to be made precise later, $u_*$ is the \emph{unique} positive and finite solution  to \eqref{eq-02}, hence showing that the Schrödinger system \eqref{eq-01} admits a unique solution.

\subsection*{Some ideas of Fortet's proof}

Let us   sketch  the proof of \eqref{eq-06}. Two consecutive applications of the monotone convergence theorem in the definition of $ \Phi[u_n]$ show that in virtue of \eqref{eq-04}
\begin{enumerate}[(i)]
\item
for each $y\in\YY,$ the sequence $v_n(y):= \IX p(x',y)u_n(x') ^{ -1} \,\mu (dx')$ is increasing and $\Lim n v_n(y)= \IX p(x',y) u_*(x')^{ -1}\,\mu (dx')=:v_*(y)\in[0, \infty]$;
\item
for each $x\in\XX,$ the sequence $ \Phi[u_n](x)
	=\IY p(x,y)v_n(y)^{ -1}\,\nu (dy) $
is decreasing and 
\begin{align}\label{eq-12}
\Lim n  \Phi[u_n](x)
	=\IY p(x,y)v_*(y)^{ -1}\,\nu (dy)
	= \Phi[u_*](x)\in[0, \infty], \quad x\in\XX.
\end{align}
\end{enumerate}
Passing to the limit in \eqref{eq-03} gives
\begin{align}\label{eq-05}
u_*= \Phi[u_*]\wedge U\le \Phi[u_*],
\end{align}
which is not enough to ensure \eqref{eq-06}. 
However,  for any positive function $u$ on $\XX$ we have
\begin{align}\label{eq-21}
\begin{split}
\IX \frac{ \Phi[u](x)}{u(x)}\ \mu (dx)
	&= \int _{ \XY} \frac{p(x,y) }{u(x) \IX p(x',y)u(x')^{ -1}\,\mu (dx')}\ \mu (dx)\nu (dy)\\
	&= \IY  \frac{\IX p(x,y)u(x)^{ -1}\,\mu (dx)}{\IX p(x',y)u(x')^{ -1}\,\mu (dx')}\ \nu (dy)
	=\IY \nu (dy)\\&=1.
\end{split}
\end{align}
In particular, once it is proved  that $u_*$ does not vanish, we obtain \\ $ \displaystyle{\IX  \Big(\underbrace{\frac{ \Phi[u_*](x)}{u_*(x)}-1} _{ \ge 0}\Big)\ \mu (dx)=0}$ and it follows with \eqref{eq-05} that \eqref{eq-06} is satisfied.

Now, we turn to a rigorous treatment of this problem.

\section{Preliminary results}\label{sec-prelim}

Of course, one must be careful to ensure that all the above quantities are well defined: there are several measurability and integrability issues to be considered. Note also that it is important to verify that the fixed point $u_*$ does not vanish. The remainder of the present section is dedicated to the establishment of several preliminary results for the proof of Theorem \ref{res-11}.

\subsection*{Conventions about infinite numbers and nonnegative integrals}

For any number $s\in[0, \infty]$, we set $s ^{ -1}:= \left\{ \begin{array}{ll}
s ^{ -1},& \textrm{if } 0<s< \infty,\\
0,& \textrm{if }s= \infty,\\
\infty,& \textrm{if }s=0.
\end{array}\right. $ 
\\
It is understood that any integral of a measurable $[0, \infty]$-valued function is univoquely well-defined as a $[0, \infty]$-valued integral, by means of Beppo Levi's monotone convergence theorem. More precisely, for any measurable  function $f: \mathcal{Z}\to[0, \infty]$ and any $ \sigma$-finite nonnegative measure $m$ on $ \mathcal{Z},$ we define 
\begin{align*}
\int _{ \mathcal{Z}} f\,dm:=\sup _{ k\ge 1}\int _{ \mathcal{Z}_k} f\wedge k\, dm\in[0, \infty],
\end{align*} 
where $ \mathcal{Z}_k\uparrow \mathcal{Z}$ with $m( \mathcal{Z}_k)< \infty$ for all $k$. 
\\
Let 
$f: \mathcal{Z}\to[0, \infty)$ and $g: \mathcal{Z}\to[0, \infty]$ be  nonnegative   measurable functions with $f$ finite and $g$  taking possibly  infinite values. In order to prevent from the undetermined product $0\times \infty,$ we set  
\begin{align*}
\int _{ \mathcal{Z}} fg\, dm:=\int _{ \left\{f>0\right\} } fg\, dm.
\end{align*}
This amounts to decide that $f(z)=0\implies f(z)g(z)=0,$ even if $g(z)= \infty.$ This convention is justified by the limiting procedure
\begin{align*}
\int _{ \mathcal{Z}} fg\, dm
	:=\sup _{ k\ge 1}\int _{ \mathcal{Z}} f\ (g\wedge k)\, dm
	=\sup _{ k\ge 1}\int _{ \left\{f>0\right\} } f\ (g\wedge k)\, dm
\end{align*}
and $0\times ( \infty\wedge k)=0.$

\subsection*{Measurability issues}

Since the unknown measures  $\aa$ and $\bb$ are required to be $ \sigma$-finite, the integrals $x\mapsto\IY p(x,y)\, \bb(dy)$ and $y\mapsto\IX p(x,y)\, \aa(dx)$ in \eqref{eq-01} are well-defined numbers in $[0, \infty].$ Furthermore, as $p$ is assumed to be jointly measurable on $\XY,$ Fubini-Tonelli theorem ensures that these $[0, \infty]$-valued functions are measurable.
\\
For any $[0, \infty]$-valued function $u$ on $\XX,$ we introduce
\begin{align*}
\Psi[u](y):=\IX p(x',y)u(x')^{ -1}\,\mu(dx')\in[0, \infty],\qquad y\in\YY,\ \nu\ae
\end{align*}
so that the Fortet mapping writes as 
\begin{align*}
\Phi[u](x)=\IY p(x,y) \Psi[u](y)^{ -1}\,\nu(dy)\in[0, \infty],
\qquad x\in\XX,\ \mu\ae
\end{align*} 
These  are well-defined integrals once the above conventions are adopted.  
Again, Fubini-Tonelli theorem ensures that these $[0, \infty]$-valued functions are measurable.

\subsection*{Convergence of the algorithm}

From now on, any function is assumed to be measurable  without further mention.

\begin{lemma}[Monotonicity of $\Psi$ and $\Phi$] \label{res-03}
For any $[0, \infty]$-valued functions $u,u',$
\begin{align*}
u\le u',\ \mu\ae\implies \Psi[u]\ge\Psi[u'],\ \nu\ae\implies \Phi[u]\le\Phi[u'],\ \mu\ae
\end{align*}
\end{lemma}

\begin{proof}
It is a direct consequence of the nonnegativity of the integrands.
\end{proof}

Let us revisit the sloppy argument that was proposed in view of justifying \eqref{eq-12}.  

\begin{lemma}[Convergence of the algorithm] \label{res-04}\ 
\begin{enumerate}[(a)]
\item

The sequence $(u_n)$  defined by \eqref{eq-03}  decreases: it satisfies \eqref{eq-04}, and its pointwise limit 
\begin{align*}
u_*:=\Lim n u_n
\end{align*}
exists and satisfies 
$
0\le u_*\le U.
$
\item
The sequence $(\Phi[u_n])$  decreases: for all $n\ge 1,$ $0\le \Phi[u_{ n+1}] \le \Phi[u_n]\le \Phi[U].$

\item
Under the additional assumption:
\begin{align}\label{eq-14}
\Phi[U](x)< \infty,\quad \forall x\in\XX,\ \mu\ae,
\end{align}
\begin{enumerate}[(i)]
\item
for any $n\ge1$, $\Phi[u_n]< \infty,\ \mu\ae,$
\item
the identity \eqref{eq-12} is valid, i.e.
\begin{align*}
\Lim n\Phi[u_n](x)=\Phi[u_*](x), \quad \forall x\in\XX,\ \mu\ae,
\end{align*}
\item
$u_*$ satisfies \eqref{eq-05}, i.e.
 \begin{align*}
 0\le u_*= \Phi[u_*]\wedge U\le \Phi[u_*]\le \Phi[U]< \infty,\ \mu\ae
 \end{align*}
\end{enumerate}
\end{enumerate}
\end{lemma}

\begin{proof}
\boulette{(a)}
As $\Phi$ is an increasing mapping in the sense of Lemma \ref{res-03}, the sequence $(u_n)$  satisfies \eqref{eq-04},  its pointwise limit $u_*$ exists and satisfies 
$
0\le u_*\le U.
$
\Boulette{(b)}
It follows from (a)  and the monotonicity of $\Phi$ (Lemma \ref{res-03}) again.

\Boulette{(c)}
 It also follows  from (a) and  the monotonicity of $\Phi$ that 
for $\nu$-almost all $y,$ $\Psi[u_n](y)$  increases with $n$, and by monotone convergence: 
\begin{align*}
\Lim n\Psi[u_n](y)=\Psi[u_*](y), \quad \forall y\in\YY,\ \nu\ae
\end{align*}
On the other hand, for $\mu$-almost all $x$, $\Phi[u_n](x)=\IY p(x,y)\Psi[u_n](y) ^{ -1}\,\nu(dy)$ decreases with $n$, but one cannot apply the monotone convergence theorem without assuming that for some $n_o,$ $\Phi[u_{ n_o}](x)$ is finite. Assumption \eqref{eq-14} corresponds to $n_o=1.$
\end{proof}

\subsection*{Reduction of the problem}

We show that, without loss of generality, the measures $\aa$ and $\bb$ can be chosen to be equivalent (in the sense of measure theory) to $\mu$ and $\nu$ respectively.

\begin{proposition}[Reduction of the problem] \label{res-06}
Without loss of generality, one can assume that 
\begin{equation}\label{eq-07}
\begin{array}{ll}
\supp  \mu =\XX,\qquad &\mu\ll \aa\ll\mu,\\
\supp \nu=\YY,\qquad &\nu\ll \bb\ll\nu,
\end{array}
\end{equation}
where $\supp$ stands for the topological support of a Borel measure, and also that
\begin{equation}\label{eq-08}
\left\{
\begin{array}{ccll}
\emph{(i)}&\quad  &\nu( p(x,\cdot)>0) >0,\quad &\forall x\in\XX,\ \mu\ae\\
\emph{(ii)}&&\mu( p(\cdot,y)>0)>0,\quad &\forall y\in\YY,\ \nu\ae
\end{array}
\right.
\end{equation}
\end{proposition}

\begin{proof}
To see this, note that if \eqref{eq-01} admits the solution $(\aa,\bb)$, then $\mu\ll \aa$ and $\nu\ll \bb.$  
On the other hand, denoting $p(x;\bb):=\IY p(x,y)\, \bb(dy),$ by the first identity in \eqref{eq-01}  we see that for any open set $G\subset \XX$ such that $ \mu(G)=0,$ we must have $\int_G p(x;\bb)\, \aa(dx)=0.$ This implies that $\1_G(x)p(x;\bb)=0,$ for all $x\in\XX,$ $\aa\ae$ Hence, $p(x;\bb)=0,$ for all $x\in\XX\setminus\supp\mu,$ $\aa\ae$, which is equivalent to $p(x,y)=0,$ for all $(x,y)\in(\XX\setminus\supp\mu)\times\YY,$ $\aa\otimes \bb\ae$ A similar reasoning also leads to $p(x,y)=0,$ for all $(x,y)\in\XX\times(\YY\setminus\supp\nu),$ $\aa\otimes \bb\ae$ Therefore
\begin{align}\label{eq-09}
p(x,y)=0, \ \forall (x,y)\not\in\supp\mu\times\supp\nu,\ \aa\otimes \bb\ae
\end{align}
This implies that \eqref{eq-01} is equivalent to
\begin{align*}
\left\{
\begin{array}{ll}
 \displaystyle{\aa'(dx) \IY p(x,y)\, \bb'(dy)}&=\ \mu (dx),\\ \\
\displaystyle{\bb'(dy)\IX p(x,y)\, \aa'(dx)} &=\ \nu (dy),
\end{array}
\right.
\end{align*}
where $\aa':=\1 _{ \supp \mu}\aa$ and $\bb':=\1 _{ \supp\nu}\bb.$ In other words, one can assume that $\supp \aa\subset\supp\mu$ and $\supp \bb\subset\supp\nu$ without loss of generality. Together with $\mu\ll \aa,\nu\ll \bb,$  this non-restricting convention implies that $\mu\ll \aa\ll\mu$ and $\nu\ll \bb\ll\nu.$ Finally, in view of this and \eqref{eq-09}, one is allowed to decide that $\XX=\supp\mu$ and $\YY=\supp\nu$. This completes the argument demonstrating that \eqref{eq-07}  is not a restriction.
\\
Finally, remark that under the conventions \eqref{eq-07}, it is necessary for \eqref{eq-01} to admit a solution  that \eqref{eq-08} is satisfied.
\end{proof}

\emph{From now on  the conventions  \eqref{eq-07} and \eqref{eq-08} are assumed to hold.  }

In particular $\uu:= d \mu/d\aa $ is a well-defined  positive finite function,  and the  Radon-Nikodym derivatives $\uu ^{ -1}$ and $ \Psi[\uu]^{ -1}$  appearing in \eqref{eq-10} are finite. The finiteness of $\uu$ and $\uu ^{ -1}$ follows from $\mu\ll \aa\ll\mu$ in \eqref{eq-07}, and the finiteness of   $ \Psi[\uu]^{ -1}$  is the object of next result.

\begin{lemma}\label{res-02}
For any nonnegative finite function $u,$ 
\begin{align*}
 \Psi[u](y)\in(0, \infty],\quad \forall y\in\YY,\ \nu\ae
\end{align*}
 does not vanish.
\end{lemma}

\begin{proof}
It is a direct consequence of   \eqref{eq-08}-(ii).
\end{proof}

Next lemma is a key result for the proofs of existence and uniqueness of the positive solution to $\Phi[u]=u.$

\begin{lemma}\label{res-01}
Suppose that 
\begin{align}\label{eq-15}
\Psi[U](y)< \infty,\quad \forall y\in\YY,\ \nu\ae
\end{align}
Then, for any function $u$ such that $0\le u\le U,$
\begin{align*}
\int _{ \{\Phi[u]>0\}} \frac{\Phi[u]}{u}\, d \mu= \nu(\Psi[u]< \infty).
\end{align*}
\end{lemma}

\begin{proof}
For any integer $k\ge 1,$ denoting $u^k:=U/k \vee u,$ we see that
$0<U/k\le u^k\le U,$  the sequence $(u^k)$ decreases pointwise to $u,$ and with   Lemma \ref{res-03}, Lemma \ref{res-02} and our hypothesis \eqref{eq-15}, we obtain
\begin{align}\label{eq-17}
\Psi[u^k]\le k\Psi[U]< \infty,\qquad \nu\ae
\end{align}
This being said, 
by monotone convergence, 
\begin{align*}
\int _{ \{\Phi[u]>0\}} \frac{\Phi[u]}{u}\, d \mu
	=\sup_k \IX \frac{\Phi[u]}{u^k}\, d \mu.
\end{align*}
On the other hand, for any $k\ge 1,$ with Fubini-Tonelli theorem
\begin{align*}
\IX \frac{\Phi[u]}{u^k}\, d \mu
	=\int _{ \XY}   \frac{p(x,y)}{u^k(x)}\  \Psi[u](y) ^{ -1}\, \mu(dx)\nu(dy)
	=\IY \Psi[u^k]\, \Psi[u] ^{ -1}\,d\nu,
	\end{align*}
where  the estimate \eqref{eq-17} is  necessary  to prevent from  undetermined products at those $y$ such that $\Psi[u](y) ^{ -1}=0$, also note that by Lemma \ref{res-02}, $\Psi[u] ^{ -1}< \infty,\ \nu\ae$ 
With the monotonicities  of the mapping $\Psi$ and  the sequence $(u^k)$, we see with Beppo-Levi theorem that $ \Psi[u^k]\, \Psi[u] ^{ -1}$ is an increasing sequence converging pointwise to $\1 _{\{ \Psi[u]< \infty\}}.$ Therefore
\begin{align*}
\sup_k \IX \frac{\Phi[u]}{u^k}\, d \mu
	=\nu(\Psi[u]< \infty).
\end{align*}
The announced result is a direct consequence of these  considerations.
\end{proof}

\begin{remark} 
The assumption $\Phi[U]< \infty$ requires that $U$ is not too large, while $\Psi[U]< \infty$ forces $U$ not to be too close to zero.
% These two growth controls ensure that $U $ has ``the same order of magnitude as'' $d\mu/d\aa $ where $\aa$ is part of the solution $(\aa,\bb)$. Indeed, recall that $\mu= u_* \,\aa,$ $0<u_*\le U$ and $u_1=U$  initializes the algorithm \eqref{eq-03} converging to $u_*$. 
\end{remark}

Next result is merely a remark, but it is of practical use.

\begin{lemma}[From almost everywhere to everywhere] \label{res-10}
Let the nonnegative function  $u$ be defined $\mu\ae$ and satisfy $u=\Phi[u],\ \mu\ae$  Then, the function $\bar u:=\Phi[u]$ is defined everywhere and satisfies $\bar u=\Phi[\bar u],$ everywhere.
\end{lemma}

\begin{proof}
For any $y$,
\begin{align*}
\Psi[\bar u](y):=\IX p(x',y) \bar u(x') ^{ -1}\, \mu(dx')
	=\IX p(x',y)  u(x') ^{ -1}\, \mu(dx')
	=:\Psi[u](y),
\end{align*}
because $\bar u:=\Phi[u]=u,\ \mu\ae$ It follows that for all $x$,
\begin{align*}
\Phi[\bar u](x)
	:= \IY p(x,y) \Psi[\bar u](y) ^{ -1}\, \nu(dy)
	= \IY p(x,y) \Psi[ u](y) ^{ -1}\, \nu(dy)
	=:\Phi[u](x)=:\bar u(x),
\end{align*}
which is the announced result.
\end{proof}

\section{Solving Fortet's equation}\label{sec-main}

We call  equation \eqref{eq-02}:  $\Phi[u]=u,$ Fortet's equation.

\subsection*{Uniqueness of the positive solution to $\Phi[u]=u$}
This result is a corollary of Lemma \ref{res-01}.

\begin{proposition}[Uniqueness]\label{res-20}
Under the assumptions \eqref{eq-14} and \eqref{eq-15} and the additional requirement that
\begin{align}\label{eq-35}
\mu( \left\{x:p(x,y)>0,\forall y\right\})>0,
\end{align}
the fixed point equation $\Phi[u]=u$ admits at most one positive solution, up to the trivial transformation $u\rightarrow  \kappa u$ with $\kappa>0.$
\end{proposition}

\begin{proof}
Let $u'$ and $u''$ be two positive solutions: $0<u'=\Phi[u']\le U$ and $0<u''=\Phi[u'']\le U.$ Consider their maximum: $u:=u'\vee u''.$ It satisfies $0<u\le U$ and by monotonicity of $\Phi,$
\begin{align*}
\Phi[u]\ge \Phi[u']\vee \Phi[u'']=u'\vee u''= u>0,
\end{align*}
that is: $\Phi[u]/u\ge 1.$
On the other hand, by Lemma \ref{res-01} 
\begin{align*}
\IX \frac{\Phi[u]}{u}\, d \mu=\int _{ \{\Phi[u]>0\}} \frac{\Phi[u]}{u}\, d \mu=
 \nu(\Psi[u]< \infty)\le 1,
\end{align*}
where we took advantage of $ \left\{\Phi[u]>0\right\} =\XX$ holding under out assumption \eqref{eq-35}.
It follows that 
\begin{align}\label{eq-27}
\Phi[u]=u,\ \mu\ae
\end{align}
and $\Psi[u]< \infty,\ \nu\ae$
\\
Pick an $x_o$ such that $\Phi[u](x_o)=u(x_o)$ and $p(x_o,y)>0, \forall y, \nu\ae$ Normalize $u'$ and $u''$ such that $u'(x_o)=u''(x_o)$ (this is possible since $\Phi$ is 1-homogeneous). Once this is done, we see that $u(x_o)=u'(x_o)=u''(x_o)=\Phi[u](x_o)=\Phi[u'](x_o)=\Phi[u''](x_o).$ 
\\
Since $u'\le u,$ we have $\Psi[u']\ge \Psi[u].$ Together with $$\IY p(x_o,y)\Psi[u](y) ^{ -1}\,\nu(dy)=\Phi[u](x_o)=\Phi[u'](x_o)=\IY p(x_o,y)\Psi[u'](y) ^{ -1}\,\nu(dy)$$ and the  positivity of $p(x_o,\cdot),$  this leads us to $\Psi[u]=\Psi[u'],\ \nu\ae,$ which in turns implies that $\Phi[u]=\Phi[u'],$ everywhere. This shows with \eqref{eq-27} that $u=u',\ \mu\ae$ By a similar reasoning, we also obtain $u=u'',\ \mu\ae,$ and it follows that $u'=u'',\ \mu\ae$ This implies  that $\Phi[u']=\Phi[u'']$ everywhere, which is the desired identity: $u'=u''.$ 
\end{proof}

\subsection*{Statement of the existence result}
The positivity of the fixed point $u_*$ requires additional assumptions beside \eqref{eq-14} and \eqref{eq-15}. This will be obtained at Theorem \ref{res-11} with $U=\1$ and under the

\begin{hypotheses}\label{hyp-01}\  
\begin{enumerate}[(a)]
\item
$p(x,y)>0,\ \forall x\in\XX, y\in\YY,$
\item
$\sup _{ x\in\XX, y\in\YY} p(x,y)< \infty,$

\item
$\IY [\IX p(x,y)\,\mu(dx)]^{ -1}\,\nu(dy)< \infty.$
\end{enumerate}
\end{hypotheses}

\begin{remarks}[About the Hypotheses \ref{hyp-01}]\ 
\begin{enumerate}[-]
\item
The irreducibility hypothesis (a) is a strengthening of convention \eqref{eq-08}. 

\item
Hypothesis (c) is: $\IY \Psi[\1](y) ^{ -1}\,\nu(dy)< \infty,$  with the function $U=\1.$
\item
The hypotheses (b) and (c) together imply the growth control assumption \eqref{eq-14}: $\Phi[\1](x)< \infty,\ \forall x.$ 
\end{enumerate}\end{remarks}

\begin{theorem}\label{res-11}
Under the Hypotheses \ref{hyp-01}, the limit $u_*$ of the scheme $(u_n)$ defined by \eqref{eq-03}  with $$U=\1,$$ 
is positive  and solves the fixed point equation  \eqref{eq-02}:
\begin{align*}
0< u_*=\Phi[u_*]\le 1,\quad \textrm{everywhere}.
\end{align*}
\end{theorem}

The remainder of this section is dedicated to the proof of this result.

\subsection*{Sufficient conditions for $u_*=\Phi[u_*]$}
We prove that if $\Phi[u_*]$ is positive $\mu$-almost everywhere, then $u_*$ is a fixed point of $\Phi.$

\begin{lemma}\label{res-07}
Under the assumptions \eqref{eq-14} and \eqref{eq-15}, if 
\begin{align}\label{eq-13}
\Phi[u_*]>0,\quad \mu\ae,
\end{align}
then $\Psi[u_*]< \infty,\ \nu\ae,$ and \ 
$
0< u_*=\Phi[u_*]\le U,\ \textrm{everywhere.}
$
\end{lemma}

\begin{proof}
By Lemma \ref{res-04}: $u_*=\Phi[u_*]\wedge  U,\ \mu\ae$
Let us investigate this identity in the following two cases.
\begin{enumerate}
\item[Case 1:] $\Phi[u_*](x)\le U(x).$  We have $u_*(x)=\Phi[u_*](x).$ Hence, $\Phi[u_*](x)>0$ implies that $ \displaystyle{ \frac{\Phi[u_*](x)}{u_*(x)}}=1.$

\item[Case 2:] $\Phi[u_*](x)> U(x).$  We have  $u_*(x)=U(x)<\Phi[u_*](x).$ Hence,  $ \displaystyle{ \frac{\Phi[u_*](x)}{u_*(x)}}>1.$
\end{enumerate}
Applying Lemma \ref{res-01} with $u_*$ and splitting $\IX$ into these disjoint cases, we obtain
\begin{align*}
\nu(\Psi[u_*]< \infty)	
	&= \int _{ \left\{\Phi[u_*]>0\right\} } \frac{\Phi[u_*]}{u_*}\, d \mu\\
	&= \int _{ \left\{0<\Phi[u_*]\le U\right\} } \frac{\Phi[u_*]}{u_*}\, d \mu
		+ \int _{ \left\{\Phi[u_*]>U\right\} } \frac{\Phi[u_*]}{u_*}\, d \mu\\
	&= \mu(\Phi[u_*]>0)
		+ \int _{ \left\{\Phi[u_*]>U\right\} } \underbrace{\Big(\frac{\Phi[u_*]}{u_*}-1\Big)}_{>0}\, d \mu.
\end{align*}
Supposing that
$	%\begin{align*}
\nu(\Psi[u_*]< \infty)	\le \mu(\Phi[u_*]>0),
$	%\end{align*}
for instance if $ \mu(\Phi[u_*]>0)=1,$
we obtain: $\Phi[u_*]\le U,\ \mu\ae$ and $\nu(\Psi[u_*]< \infty)	= \mu(\Phi[u_*]>0).$ Since $\Phi[u_*]\le U,\ \mu\ae,$ we conclude with \eqref{eq-05} (see Lemma \ref{res-04}) that $u_*=\Phi[u_*],\ \mu\ae$ Finally, we obtain the "everywhere identity" with Lemma \ref{res-10}. 
\end{proof}

\subsection*{On the way to the positivity of $u_*$}

With Lemma \ref{res-07} at hand, it remains  to face the problem of the positivity of $\Phi[u_*],$ see \eqref{eq-13}. This will be done by  assuming the Hypotheses \ref{hyp-01}, and in particular that $p$ is positive everywhere.

We shall invoke identity \eqref{eq-21} in a while. Here is its rigorous  statement.

\begin{lemma}\label{res-08}
Under the assumptions \eqref{eq-14} and \eqref{eq-15},  any  $u:\XX\to(0, \infty)$ such that $c ^{ -1}U\le u\le c U$ for some $1\le c < \infty,$ satisfies
\begin{align*}
\IX \frac{\Phi[u]}{u}\, d\mu=1.
\end{align*}
In particular, $ \displaystyle{\IX \frac{\Phi[u_n]}{u_n}\, d\mu=1},$ for all $n\ge 1.$
\end{lemma}

\begin{proof}
It is a direct application of Lemma \ref{res-01}, once it is noted that \eqref{eq-14} with $u\le cU$ imply $ \left\{\Phi[u]< \infty\right\} =\XX,$ and \eqref{eq-15} with $u\ge U/c$ imply $ \left\{\Psi[u]< \infty\right\} =\YY.$
\end{proof}

\begin{lemma}\label{res-12}
If there exists some $n_o\ge 1$ such that  $\Phi[u _{ n_o}]\le U,$ then $0<u_*=\Phi[u _{ n_o+1}]=\Phi[u_*]$, everywhere on $\XX.$ 
\end{lemma}

\begin{proof}
Because $\Phi[u_n]$ decreases, we see that 
\begin{align*}
u _{ n_o+1}:=U/(n_o+1)\vee \Phi[u _{ n_o}]\wedge U
	=U/(n_o+1)\vee \Phi[u _{ n_o}]
	\ge \Phi[u _{ n_o}]
	\ge \Phi[u _{ n_o+1}].
\end{align*}
Hence, $\Phi[u _{ n_o+1}]/ u _{ n_o+1}\le 1,\ \mu\ae$ But, by Lemma \ref{res-08}: $\IX \Phi[u _{ n_o+1}]/ u _{ n_o+1}\, d \mu=1.$ Therefore, $$u _{ n_o+1}=\Phi[u _{ n_o+1}],\ \mu\ae$$ Since $u _{ n_o+1}\ge U/(n_o+1)>0$ (this is the only place where $U/(n+1)\, \vee $ plays a role in the definition \eqref{eq-03} of $(u_n)$),  we conclude with Lemma \ref{res-10} that $0<u_*=\Phi[u _{ n_o+1}]=\Phi[u_*]$, everywhere on $\XX.$
\end{proof}

\begin{lemma}[Dichotomy] \label{res-13}
Assuming  that $p>0$ everywhere, for any function $0\le u\le U,$  either  $\Phi[u](x)=0,\ \forall x\in\XX$, or\quad $\Phi[u](x)>0,\ \forall x\in\XX.$
\end{lemma}

\begin{proof}
Suppose that there is some $x_o\in\XX$ such that  $\Phi[u](x_o)=0$. Then,   
\begin{align*}
0=\Phi[u](x_o)=\IY p(x_o,y)\Psi[u](y) ^{ -1}\,\nu(dy).
\end{align*}
But  $p>0$ implies that $\Psi[u] ^{ -1}= 0,\ \nu\ae$, which in turn implies that 
 $\Phi[u](x)=0,\ \forall x\in\XX.$
\end{proof}

To go further,  we need the additional Hypotheses \ref{hyp-01}.

\begin{lemma} \label{res-09}\ 
Under the  Hypotheses  \ref{hyp-01}, if there exists $x_o$ such that $\Phi[u_*](x_o)=0,$ then the sequence $(\Phi[u_n])$ converges uniformly to $\Phi[u_*]=0.$
\end{lemma}

\begin{proof}
Denote $ \gamma_u(dy):= \Psi[u](y) ^{ -1}\, \nu(dy)$, so that
$\Phi[u](x_o)=\IY p(x_o,y)\, \gamma_u(dy).$ By monotonicity of $(u_n)$ and $\Psi$, 
\begin{align*}
0\le \gamma _{ u _{ n+1}}\le \gamma _{ u_n} \le \gamma _{ \1},
\quad \forall n\ge1.
\end{align*}
Let $x_o$ be such that $\Phi[u_*](x_o)=0.$ Then, $0=\Lim n \Phi[u_n]( x_o)
=\Lim n \IY p(x_o,y)\, \gamma _{ u_n}(dy).$
\\
For any $k\ge 1,$ define $\YY_k:= \left\{p(x_o,\cdot)\ge 1/k\right\} .$
 Since $\IY p(x_o,\cdot) \, d \gamma _{ \1}=:\Phi[\1](x_o)< \infty,$ by dominated convergence
\begin{align*}
\Lim n \gamma _{ u_n}(\YY_k)\le k \Lim n \IY p(x_o,\cdot)\, d \gamma _{ u_n}=0,\quad \forall k\ge 1.
\end{align*}
On the other hand,  Hypothesis \ref{hyp-01}-(a) implies that the sequence $(\YY_k^c)$ decreases  to the empty set, and Hypothesis \ref{hyp-01}-(c) states that $ \gamma _{ \1}$ is a finite measure. It follows that  $\Lim k \gamma _{ \1}(\YY_k^c)=0$. We see with these considerations and the Hypothesis \ref{hyp-01}-(b): $\sup p< \infty,$ that  
\begin{align*}
\sup_x \Phi[u_n](x)
	=\sup_x \IY p(x,y)\, \gamma _{ u_n}(dy)
	\le \sup p\  [\gamma _{ u_n}(\YY_k)+ \gamma _{ \1}(\YY_k^c)]
\end{align*}
can be made arbitrarily small for $k$ and (then)  $n$ large enough.
\end{proof}

We are ready to complete the proof of the main result.

\subsection*{Completion of the proof of Theorem \ref{res-11}}

Note that the Hypotheses \ref{hyp-01} imply $\Psi[U]< \infty$ and $\Phi[U]< \infty$ with $U=\1.$ 
\\
By Lemma \ref{res-07}, it is sufficient to prove that $\Phi[u_*]>0$ everywhere. Let us do it ad absurdum. 
\\
Suppose that $\Phi[u_*](x_o)=0,$ for some $x_o.$   We know with Lemma \ref{res-09} that $\Phi[u_n]$ converges uniformly to zero. Therefore, there exists some $n_o\ge 1$ such that  $\Phi[u _{ n_o}]\le \1.$ Applying Lemma \ref{res-12}, we obtain:   $0<u_*=\Phi[u _{ n_o+1}]=\Phi[u_*]$, everywhere on $\XX,$ which contradicts our assumption that $\Phi[u_*](x_o)=0$ for some $x_o$, and completes the proof of  Theorem \ref{res-11}. 

\section{A first variant}

Observe that Lemma \ref{res-09} is the only place where other assumptions than  \eqref{eq-14} and \eqref{eq-15} are required. Any variant of Fortet's proof must focus on an analogue of this lemma. 

As in the Hypotheses \ref{hyp-01}-(a) and (b) of Theorem \ref{res-11}, $p$ is assumed to be positive and upper bounded, but 
Hypothesis \ref{hyp-01}-(c) is replaced  by the Hypotheses \ref{hyp-02}.

\begin{proposition}\label{res-17}
Suppose as in Theorem \ref{res-11} that $p$ is positive and bounded. Under  the additional Hypotheses \ref{hyp-02}, the limit $u_*$ of the scheme $(u_n)$ defined by \eqref{eq-03}  with $U=\1,$
is positive  and solves the fixed point equation  \eqref{eq-02}:
$
0< u_*=\Phi[u_*]\le 1, \textrm{everywhere}.
$
\end{proposition}

The remainder of this section is dedicated to the proof of Proposition \ref{res-17}. We essentially follow Fortet.

\subsection*{Twisting the transition density}\label{twist}

Next result is merely a remark. Nevertheless we state it as a proposition because it gives a way to obtain a variant of Theorem \ref{res-11}.
Let us write $(\aa,\bb)\in\sol (p,\mu,\nu)$ to state that $(\aa,\bb)$ solves the Schrödinger system \eqref{eq-01}.

\begin{proposition} \label{res-14}
Consider the new transition density
\begin{align*}
\tilde p(x,y):= \alpha(x) \beta(y) p(x,y),
\qquad x\in\XX, y\in\YY,
\end{align*}
where $ \alpha:\XX\to (0, \infty)$ and $ \beta:\YY\to (0, \infty)$ are positive. Then,
\begin{align*}
(\aa,\bb)\in\sol(p,\mu,\nu)\iff (\tilde \aa, \tilde \bb):=( \alpha ^{ -1}\aa, \beta ^{ -1}\bb)\in\sol(\tilde p,\mu,\nu).
\end{align*}
\end{proposition}

\begin{proof}
Setting  $(\tilde \aa, \tilde \bb):=( \alpha ^{ -1}\aa, \beta ^{ -1}\bb),$ we observe that the measure
\begin{align*}
\tilde \aa(dx)\tilde p(x,y)\tilde \bb(dy)
	=\aa(dx)p(x,y)\bb(dy)
\end{align*}
has the desired product form and marginal measures $(\mu,\nu).$ 
\end{proof}

As a consequence, when proving that $(\aa,\bb)$ solves \eqref{eq-01}, it is enough to show that $(\tilde \aa,\tilde \bb)\in\sol(\tilde p,\mu,\nu)$  for some well chosen $\tilde p,$ see Corollary \ref{res-15} below.
Define the mapping
\begin{align*}
\tilde \Phi[u](x)
	:= \IY \tilde p(x,y) \left[ \IX \tilde p(x',y)u(x') ^{ -1} \,\mu(dx')\right] ^{ -1} \nu(dy),
\end{align*}
which is the analogue of $\Phi$ where $p$ is replaced by $\tilde p.$   Consider  also the scheme $(\tilde u_n)$ defined by \eqref{eq-03} with $U=\1$ and $\tilde \Phi$ instead of $\Phi.$

\begin{corollary}\label{res-15}
For the fixed point equation \eqref{eq-06} to admit a positive solution, it is sufficient to find a transition density $\tilde p$ as in Proposition \ref{res-14} such that
\begin{align}\label{eq-22}
\left\{
\begin{array}{lll}
(a) &\tilde \Psi[\1](y):=\IX \tilde p(x,y)\,\mu(dx) < \infty,\quad &\forall y\in\YY, \ \nu\ae\\
(b) &\tilde \Phi[\1](x) := \IY \tilde p(x,y) \tilde\Psi[\1](y) ^{ -1}\,\nu(dy)< \infty,
\quad &\forall x\in\XX,\ \mu\ae
\end{array}
\right.
\end{align}
and
\begin{align}\label{eq-23}
\tilde\Phi[\tilde u_*](x_o)>0,\quad \textrm{for some }x_o.
\end{align}
\end{corollary}

\begin{proof}
One can apply Lemmas \ref{res-04} and \ref{res-07} with $U=\1,$ provided that 
\eqref{eq-22} is satisfied. 
If this holds, thanks to these lemmas and  Lemma \ref{res-13}, to obtain that $\tilde u_*:=\Lim n\tilde u_n$ is a positive fixed point of $\tilde\Phi:$ i.e.\ $0< \tilde u_*=\tilde\Phi[\tilde u_*],$  it is sufficient to prove  \eqref{eq-23}. We conclude with Proposition \ref{res-14}.
\end{proof}

\subsection*{A clever twisting}

Fix some compact subset $K\subset \XX$ and pick a function $ \rho$ on $\XX$ such that:
\begin{enumerate}[(i)]
\item
$ \rho:\XX\to(0, \infty)$ is positive, continuous and 
$\IX \rho(x)\, \mu(dx)=1;$
\item
$ \left\{ \rho\ge 1\right\} \subset K;$
\item
$\IX \left[ \IY p(x,y')\, \nu(dy')\right] ^{ -1}\, \rho(x)\, \mu(dx)< \infty.$
\end{enumerate}
Because of convention \eqref{eq-08}-(i), such a function  exists.
By a direct application of Theorem \ref{res-11}, we know that there exists a solution $(\bar \aa, \bar\bb)$ to the Schrödinger system
\begin{align*}
\left\{
\begin{array}{ll}
 \displaystyle{\bar \aa(dx) \IY p(x,y)\, \bar \bb(dy)}&=\ \rho(x)\,\mu (dx),\\ \\
\displaystyle{\bar \bb(dy)\IX p(x,y)\, \bar \aa(dx)} &=\ \nu (dy).
\end{array}
\right.
\end{align*}
Consider the twisted  transition density 
\begin{align*}
\tilde p(x,y):= \frac{d\bar \aa}{d \mu}(x) p(x,y).
\end{align*}
Note that   convention  \eqref{eq-07} and $ \rho>0 $ imply:  $\mu\ll \bar \aa\ll\mu$ and $\nu\ll\bar\bb\ll\nu.$

\begin{lemma}\label{res-24}
With this choice of $\tilde p,$ the estimates  \eqref{eq-22} are satisfied and 
\begin{align}\label{eq-24}
\tilde \Phi[\1]= \rho.
\end{align}
\end{lemma}

\begin{proof}
We have: $\tilde \Psi[\1](y)=\IX  p(x,y)\,\bar \aa(dx).$ As $\IY\tilde\Psi[\1](y)\, \bar \bb(dy)=\nu(\YY)=1< \infty,$ we see that $\tilde\Psi[\1]$ is finite $\bar \bb\ae$, hence  $ \nu\ae$, which is \eqref{eq-22}-(a).
On the other hand, \eqref{eq-22}-(b) is also satisfied because for any $x,$
\begin{align*}
%\begin{split}
\tilde\Phi[\1](x)
	&= \frac{d\bar\aa}{d \mu}(x) 
		\IY p(x,y) \left[ \IX p(x',y) \, {\bar\aa}(dx')\right] ^{ -1} \nu(dy)
	= \frac{d\bar\aa}{d \mu}(x)
		\IY p(x,y)\,\bar \bb(dy)\\
	&= \rho(x)< \infty,
%\end{split}
\end{align*}
which completes the proof.
\end{proof}

\subsection*{Completion of the proof of Proposition \ref{res-17}}
One more preliminary result is necessary.

\begin{lemma}\label{res-16}
For \eqref{eq-23} to be verified, it is sufficient that
\begin{align}\label{eq-25}
\tilde\Phi[\tilde u_n] \textrm{ is continuous on }K ,\quad \textrm{for all } n\ge 1.
\end{align}
\end{lemma}

\begin{proof}
Two exclusive cases are considered: (a) there exists $n_o\ge 1$ such that $\{\tilde\Phi[\tilde u _{ n_o}]< 1\}=\XX,$ or (b) for all $n\ge 1,$   $\{\tilde\Phi[\tilde u_n]\ge 1\}\not= \emptyset.$
\begin{enumerate}[(a)]
\item
If $\tilde\Phi[ \tilde u _{ n_o}]<1$ everywhere,  then Lemma \ref{res-12} tells us that 
$0<\tilde u_*=\tilde\Phi[\tilde u_*]
$
everywhere.
\item
In the alternate case, for each $n\ge 1,$  $A_n:= \left\{\tilde\Phi[\tilde u_n]\ge 1\right\}$ is nonempty.
\\
 As $(\tilde u_n)$ is decreasing and upper bounded by $\1$ and $\tilde\Phi$ is increasing, we see with \eqref{eq-24} that $(A_n)$ is a decreasing sequence of  subsets of $ \left\{\tilde\Phi[\1]\ge 1\right\} = \left\{ \rho\ge 1\right\}.$  By construction of $\tilde\Phi,$ we have chosen   $ \rho$ such that  $\left\{ \rho\ge 1\right\}\subset K.$ It follows that 
\begin{align*}
A _{ n+1}\subset A_n\subset K,\quad A_n\not=\emptyset,\quad \forall n\ge 1.
\end{align*}
If in addition, we suppose  that \eqref{eq-25} holds,
we observe that $(A_n)$ is a decreasing sequence of  nonempty compact sets. Hence its limit $A _{ \infty}=\cap _{ n\ge 1}A_n= \left\{ \tilde\Phi[\tilde u_*]\ge 1\right\} $ is nonempty. This precisely means  that \eqref{eq-23} holds for any $x_o$ in $A _{ \infty}.$
\end{enumerate}
This complete the proof of the lemma.
\end{proof}

\begin{proof}[Proof of Proposition \ref{res-17}]

We see with \eqref{eq-24}, that 
\begin{align*}
\IY p(x,y)\,\bar \bb(dy)=
\IY p(x,y) \bar \bb(y)\, \nn(dy)=\IY p(x,y) \tilde\Psi[\1] ^{ -1}(y)\, \nu(dy)=\bar u(x) \rho(x)< \infty,
\end{align*} 
for all $x\in\XX.$ Under our Hypotheses \ref{hyp-02}, this implies that 
\begin{align}\label{eq-26}
\left\{ p(x, \cdot) \tilde\Psi[\1] ^{ -1} ; x\in K\right\}  \textrm{ is uniformly $\nu$-integrable.}
\end{align} 
Because of the assumed continuity of $x\mapsto p(x,y),$ we obtain that $x\mapsto \IY p(x,y)\,\bar \bb(dy)=\bar u(x) \rho(x)$ is continuous on $K$, and it follows with the assumed continuity of $ \rho,$ that $d \bar\aa/ d \mu= \bar u ^{ -1}$ is also continuous on $K.$
\\
On the other hand, for all $n\ge 1,$
\begin{align*}
\tilde \Phi[\tilde u_n](x)
	:= \frac{d\bar\aa}{d \mu}(x) \IY  p(x,y) \tilde\Psi[\tilde u_n](y) ^{ -1} \nu(dy),
\end{align*}
is finite for all $x\in\XX$ because $\tilde\Phi[\tilde u_n]\le\tilde\Phi[\1]= \rho$, and so is $\IY  p(x,y) \tilde\Psi[\tilde u_n](y) ^{ -1} \nu(dy)$. But $\tilde\Psi[\tilde u_n] ^{ -1}\le \tilde\Psi[\1] ^{ -1}$ and \eqref{eq-26} imply that $ \left\{ p(x,\cdot) \tilde\Psi[\tilde u_n]^{ -1} ; x\in K\right\} $ is uniformly $\nu$-integrable, which implies that $\tilde\Phi[\tilde u_n]$ is continuous on $K.$
We conclude with Corollary \ref{res-15} and  Lemma \ref{res-16}.
\end{proof}

\subsection*{Proof of Corollary \ref{res-21}} \label{pc21}

It is enough to verify the Hypotheses \ref{hyp-02} with $$K:= \left\{x\in\Rn; |x|\le 1\right\}, $$  the unit ball. 
With the assumed  monotonicity of $ \theta$, we see that
\begin{align*}
[y-x|\ge|y-\bar x|\ge L
\implies p(x,y)\le p(\bar x,y).
\end{align*}
As
\begin{align*}
[y-x|\ge |y-\bar x|
\iff (\bar x-x)\scal \Big(y- \frac{x+\bar x}{2}\Big)\ge 0,
\end{align*}
 is the inequation (in $y$) of a half space, for any $\bar x$ in  the sphere $S:= \left\{ x: | x|=2\right\} $ centered at $0$ with radius 2, the set 
 \begin{align*}
 \YY(\bar x):= \bigcap _{ x\in K} \left\{y \in\Rn: [y-x|\ge |y-\bar x| \right\} 
 \end{align*}
 is a  convex set.  One can check that it contains an affine cone with vertex $3\bar x/4$ generated by a spherical cap $C(\bar x)$ on $S,$ centered at $\bar x$ with some positive radius. Hence,
 $\left\{y:|y|\ge 2\right\} \subset  \bigcup _{ \bar x\in S}\YY(\bar x),$
  and extracting a finite covering of the compact  $S$ from  $\{C(\bar x); \bar x\in S\}, $  we can pick finitely many points $\bar x_1,\dots,\bar x_n$ in $S$ such that
  \begin{align*}
\left\{y:|y|\ge 2\right\} \subset  \bigcup _{ i\le n}\YY(\bar x_i).
 \end{align*}
 On the other hand, because $ \theta$ is positive, bounded and continuous, the ratio\\ $c:=\sup_{ x\le 1, y\le 2+L} p(x,y)/ \inf_{ x\le 1, y\le 2+L} p(x,y)$ is finite and
\begin{align*}
\sup _{ x\in K}p(x,y)\le   c\ p(0,y),\quad \forall y: |y|\le 2+L.
\end{align*}
Note that any $x\in K$ instead of $x=0$ would do the job in this inequality. Putting everything together, we see that
\begin{align*}
\sup _{ x\in K} p(x,y)\le c\ p(0,y)+ \sum _{ i\le n}p(\bar x_i,y),
\quad \forall y\in \Rn,
\end{align*}
 demonstrating the desired property.
\hfill$\Box$

\section{A second variant}

We look at another way to replace Lemma \ref{res-09}. Again $p$ is positive and the hypotheses  \eqref{eq-14} and \eqref{eq-15}  are assumed to hold, but no topology is required and the Hypotheses \ref{hyp-01} are replaced by the Hypotheses \ref{hyp-03}.

\begin{lemma}\label{res-19}
Under the  Hypothesis  \ref{hyp-03}-(c), if there exists $x_o$ such that   $\Phi[u_*](x_o)=0$ and $p(x_o,y)>0, \forall y, \nu\ae,$ then the sequence $(\Phi[u_n]/U)$ converges uniformly to $0.$
\end{lemma}

\begin{proof}
As in the proof of Lemma \ref{res-09}, denote $ \gamma_u(dy):= \Psi[u](y) ^{ -1}\, \nu(dy)$, so that
$\Phi[u](x)=\IY p(x,y)\, \gamma_u(dy).$ By monotonicity of $(u_n)$ and $\Psi$, 
$	%\begin{align*}
0\le \gamma _{ u _{ n+1}}\le \gamma _{ u_n} \le \gamma _{ U},
\ \forall n\ge1.
$	%\end{align*}
\\
Let $x_o$ be such that $\Phi[u_*](x_o)=0$ and $r_*>1$ be the conjugate number of $r$ appearing at Hypothesis \ref{hyp-03}-(c): $1/r+1/r_*=1.$ By Hölder's inequality, for any $n\ge 1$ and any $x,$ 
\begin{align*}
\Phi[u_n](x&)=\IY p(x,y)\, \gamma _{ u_n}(dy)
	=\IY \frac{p(x,y)}{p(x_o,y)}\ p(x_o,y) ^{ 1/r} p(x_o,y) ^{ 1/r_*}\, \gamma _{ u_n}(dy)\\
	&\le \left(\IY p(x_o,y)\, \gamma _{ u_n}(dy)\right) ^{1/ r_*}
		\left(\IY \left(\frac{p(x,y)}{p(x_o,y)}\right) ^r\ p(x_o,y) \, \gamma _{ U}(dy)\right) ^{ 1/r},
\end{align*}
where we used $ \gamma _{ u_n}\le \gamma_U.$ We see with our hypothesis that
\begin{align*}
\sup_x \Phi[u_n]/U(x)
	\le c ^{ 1/r}\Phi[u_n](x_o) ^{ 1/r_*} \underset{n\to \infty}{\longrightarrow} 0,
\end{align*}
which is the announced result.
\end{proof}

\begin{proposition}\label{res-18}
Under the Hypotheses \ref{hyp-03}, the limit $u_*$ of the scheme $(u_n)$ defined by \eqref{eq-03}  is positive  and solves the fixed point equation  \eqref{eq-02}:
\begin{align*}
0< u_*=\Phi[u_*]\le U,\quad \textrm{everywhere}.
\end{align*}
\end{proposition}

\begin{proof}
By Lemma \ref{res-07}, it is sufficient to prove that $\Phi[u_*]>0$ everywhere. Let us do it ad absurdum. 
\\
Suppose that $\Phi[u_*](x_o)=0,$ for some $x_o.$   We know with Lemma \ref{res-19} that $\Phi[u_n]/U$ converges uniformly to zero. Therefore, there exists some $n_o\ge 1$ such that  $\Phi[u _{ n_o}]\le U.$ Applying Lemma \ref{res-12}, we obtain:   $0<u_*=\Phi[u _{ n_o+1}]=\Phi[u_*]$, everywhere on $\XX,$ which contradicts our assumption that $\Phi[u_*](x_o)=0$ for some $x_o.$
\end{proof}

%\bibliographystyle{siam}
%\bibliographystyle{alpha}
%\bibliography{bib-christian.bib}

\end{document}